\documentclass[reqno, 11pt]{amsart}
\usepackage{mathtools}
\usepackage{amsmath}
\usepackage{amssymb}
\usepackage{yhmath}
\usepackage{graphicx}
\usepackage{ mathrsfs }
\usepackage{bbm}
\usepackage{xcolor}
\usepackage{tikz-cd}
\DeclareMathAlphabet{\mathpzc}{OT1}{pzc}{m}{it}

\usepackage{tikz-cd}
\usepackage{tikz}
\usetikzlibrary{patterns}
\usepackage{hyperref}
\newtheorem{theorem}{Theorem}[section]

\newtheorem*{claim*}{Claim}

\newtheorem{lemma}[theorem]{Lemma}
\newtheorem{lem}[theorem]{Lemma}

\newtheorem{cor}[theorem]{Corollary}

\newtheorem{Thm}[theorem]{Theorem}
\newtheorem{thm}[theorem]{Theorem}

\theoremstyle{definition}
\newtheorem{Def}[theorem]{Definition}

\theoremstyle{remark}

\newtheorem{rmk}[theorem]{Remark}
\newtheorem{Rmk}[theorem]{Remark}

\numberwithin{equation}{section}
\newcommand{\bi}{\begin{itemize}}
\newcommand{\ei}{\end{itemize}}

\newcommand{\op}{\operatorname}
\newcommand{\Out}{\op{Out}}

\newcommand{\bb}{\mathbb}

\newcommand{\be}{\begin{equation}}
\newcommand{\ee}{\end{equation}}
\newcommand{\Ga}{\Gamma}

\newcommand{\R}{\mathbb R}

\newcommand{\N}{\mathbb N}
\newcommand{\ga}{\gamma}

\newcommand{\La}{\Lambda}
\newcommand{\inte}{\op{int}}
\newcommand{\ba}{\backslash}

\newcommand{\cal}{\mathcal}
\newcommand{\br}{\mathbb R}
\newcommand{\SO}{\op{SO}}
\newcommand{\Isom}{\op{Isom}}

\newcommand{\F}{\cal F}

\newcommand{\bH}{\mathbb H}

\newcommand{\G}{\Gamma}
\newcommand{\m}{\mathsf{m}}

\newcommand{\T}{\op{T}}

\newcommand{\e}{\varepsilon}

\newcommand{\BMS}{\op{BMS}}
\renewcommand{\L}{\mathcal L}
\newcommand{\fa}{\mathfrak a}

\renewcommand{\S}{\mathbb S}

\newcommand{\so}{\SO^\circ(n,1)}

\newcommand{\s}{\sigma}

\begin{document}

\title[Hausdorff dimension of Directional limit sets]{Hausdorff dimension of \\Directional limit sets for self-joinings of \\hyperbolic manifolds}

\author{Dongryul M. Kim}
\address{Department of Mathematics, Yale University, New Haven, CT 06511}
\email{dongryul.kim@yale.edu}

\author{Yair N. Minsky}
\address{Department of Mathematics, Yale University, New Haven, CT 06511}\email{yair.minsky@yale.edu}

\author{Hee Oh}
\address{Department of Mathematics, Yale University, New Haven, CT 06511}
\email{hee.oh@yale.edu}
\thanks{Minsky and Oh are partially supported by the NSF}

\begin{abstract}
The classical result of Patterson and Sullivan says that for a non-elementary convex cocompact subgroup $\Gamma<\operatorname{SO}^\circ (n,1)$, $n\ge 2$, the Hausdorff dimension of the limit set of $\Gamma$ is equal to the critical exponent of $\Gamma$. In this paper, we generalize this result for self-joinings of convex cocompact groups in two ways.

Let $\Delta$ be a finitely generated group and  $\rho_i:\Delta\to \operatorname{SO}^\circ(n_i,1)$ be
a convex cocompact faithful representation of $\Delta$ for $1\le i\le k$. Associated to $\rho=(\rho_1, \cdots, \rho_k)$,
 we consider the following self-joining subgroup of $\prod_{i=1}^k \operatorname{SO}(n_i,1)$: 
$$\Gamma=\left(\prod_{i=1}^k\rho_i\right)(\Delta)=\{(\rho_1(g), \cdots, \rho_k(g)):g\in \Delta\} .$$
\begin{enumerate}
 \item Denoting by $\Lambda\subset \prod_{i=1}^k \mathbb{S}^{n_i-1}$ the limit set
  of $\Gamma$, we first
  prove that 
$$\text{dim}_{\mathrm{H}} \Lambda=\max_{1\le i\le k} \delta_{\rho_i}$$ 
where $\delta_{\rho_i}$ is the critical exponent of the subgroup $\rho_{i}(\Delta)$.
\item Denoting by $\La_u\subset \La$ the $u$-directional  limit set for each $u=(u_1, \cdots, u_k)$ in the interior of the limit cone of $\Gamma$, we obtain that for $k\le 3$, 
$$ \frac{\psi_\Gamma(u)}{\max_i u_i }\le \text{dim}_{\mathrm{H}} \Lambda_u \le   \frac{\psi_\Gamma(u)}{\min_i u_i  }$$
where $\psi_\Gamma:\br^k\to \mathbb{R}\cup\{-\infty\}$ is the growth indicator function of $\Gamma$. 
\end{enumerate}

\end{abstract}

\maketitle
\tableofcontents
\section{Introduction}
Let $k \ge 1$, and $G = \prod_{i = 1}^k G_i$ where $G_i = \SO^{\circ}(n_i, 1)$ for $n_i \ge 2$. Consider the hyperbolic $n_i$-space $( X_i=\bH^{n_i}, d_i)$  of constant curvature $-1$.  The Lie group $G$ is the identity component of $\Isom (X)$, where $(X,d)$
is the Riemannian product  $X=\prod_{i=1}^k X_i$ with
\be \label{Riem}
d((x_i), (y_i)) = \sqrt{\sum_{i = 1}^k d_i(x_i, y_i)^2}.
\ee
 The Furstenberg boundary of $G$ is then the Riemannian product $\F = \prod_{i = 1}^k \S^{n_i - 1}$ of the geometric boundaries $\partial X_i\simeq \S^{n_i-1}$.
We consider a particular class of discrete subgroups of $G$, constructed as follows. 
Let $\Delta$ be a finitely generated group. For $1 \le i \le k$, let $\rho_i : \Delta \to G_i$ be a 
convex cocompact faithful representation of $\Delta$. 
Let 
$$\G = \left(\prod_{i = 1}^k \rho_i \right)(\Delta) =\{ (\rho_1(\sigma), \cdots, \rho_k(\sigma))\in G: \sigma\in \Delta\}.$$

We will assume that each $\rho_i(\Delta)$ is Zariski dense in $G_i$ and no two $\rho_i$'s are conjugate to each other; this implies that $\G$ is Zariski dense in $G$.
The  quotient
$\Gamma \backslash X$ is a locally symmetric Riemannian manifold of rank $k$, which we  call a self-joining of a hyperbolic manifold. Unless $k=1$ and $\rho_1(\Delta)$ is a cocompact lattice of $G_1$, $\Ga\ba X$ is of infinite volume.

For each $i$, we fix a basepoint $o_i \in X_i$.  Denote by $\La_{\rho_i}\subset\S^{n_i - 1} $
the limit set of $\rho_i(\Delta)$, which is the set of all accumulation points of the orbit $\rho_i(\Delta)o_i$ in the compactification $X_i\cup \S^{n_i - 1}$. We also denote by $\delta_{\rho_i}$ the critical exponent of
${\rho_i}(\Delta)$, which is 
 the abscissa of convergence of the Poincar\'e series $\sum_{\sigma\in \Delta}e^{-s d_i(\rho_i(\sigma) o_i, o_i)}$ (that is, the infimum of the set of $s$ for which the series converges). These two notions are independent of the choice of $o_i\in X_i$.
A well-known theorem of Patterson \cite{Pa} and Sullivan \cite{Su} says that $\delta_{\rho_i}$  is equal to the Hausdorff dimension
of the limit set  $\La_{\rho_i}$:
$$\delta_{\rho_i}=\dim \La_{\rho_i} .$$

The main aim of this paper is to investigate a higher rank analogue of this theorem. 
Let $o = (o_1, \cdots, o_k) \in X$. The limit set of $\Ga$ is the set of all accumulation points of an orbit $\Ga o $ in $\F$:
\be\label{dll} \La=\left\{ (\xi_1, \cdots, \xi_k)\in \F: \begin{matrix}
\text{$\exists$ a sequence $\sigma_\ell\in \Delta$ s.t. $\forall$ $1 \le i \le k$,} \\
\xi_i=\lim_{\ell\to \infty} \rho_i(\sigma_\ell)(o_i)
\end{matrix} \right\} .\ee

The {Hausdorff dimension} of a subset $S\subset \F$, which will be denoted by $\dim S$, is computed with respect to the Riemannian product metric of the spherical metrics on $\S^{n_i-1}$, $1\le i\le k$.
 
\subsection*{Hausdorff dimension of the limit set} 

	Our first result is the following (Theorem \ref{m0}):
\begin{thm}\label{main} We have
  \be\label{in2}  \dim{} \La =\max_{1\le i\le k}
  \delta_{\rho_i}.\ee
\end{thm}

Note that $\max_{1\le i\le k}
  \delta_{\rho_i}=\delta_{\min}$ where
   $\delta_{\min}$ denotes  the abscissa of the convergence of the series $\sum_{(\ga_1, \cdots, \ga_k) \in \Ga} e^{-s \min_i d_{i}(\ga_i o_i, o_i)}$ (see the proof of Theorem \ref{m0}).
   If  $\delta$
   denotes the critical exponent of $\Ga$ with respect to the Riemannian metric $d$ on $X$, 
   then  $\sqrt k \delta \le \delta_{\min}$ and moreover 
   $\sqrt{k}\delta=\delta_{\min}$ if and only if $\rho_i$'s are all conjugate to each other \cite{KMO}.
 It is therefore interesting to note that for $k\ge 2$, $\dim \La$ is not in general equal to $\delta$, in contrast to $k=1$ case.

  \subsection*{Limit cone and Growth indicator function}
   We also obtain estimates on the Hausdorff dimension of  directional  limit sets of $\Gamma$.  To state the estimates, we need to recall the notion of the Cartan projection, the limit cone of $\Ga$ and the growth indicator function of $\Ga$.

For $g = (g_1, \cdots, g_k) \in G$, the Cartan projection of $g$ is a vector-valued distance function: $$\mu(g) = (d_1(g_1 o_1, o_1), \cdots, d_k(g_k o_k, o_k)) \in \br_{\ge 0}^k;$$ note that the standard Euclidean norm $\| \mu(g) \|$ is equal to $d(go, o)$.

 The limit cone $\L$ of $\Ga$ is defined as the asymptotic cone of $\mu(\Ga)$, i.e.,
 $$\L=\{\lim_{i\to \infty} t_i \mu(\ga_i) \in \br_{\ge 0}^k: t_i\to 0, \ga_i\in \Ga\}.$$
  This notion was introduced by Benoist, who also showed that $\L$ is a convex cone with non-empty interior \cite{Ben}.
 
Following Quint \cite{Quint1}, the growth indicator function $\psi_\Ga : \br^k \to\bb R\cup\{-\infty\}$  is defined as follows: for an open cone $\cal C$ in $\br^k$, let $\tau_{\cal C}$ denote the abscissa of convergence of  $\sum_{\ga\in\Ga,\,\mu(\ga)\in\cal C}e^{-s d(\ga o, o)}$. Now 
for any non-zero $u\in \br^k$, let \begin{equation}\label{grow3}\psi_{\Gamma}(u):=\|u\|
\inf_{u\in\cal C} \tau_{\cal C}\end{equation}
where the infimum is taken over all open cones $\cal C$ containing $u$, and let $\psi_\Gamma(0)=0$.
It is immediate that
 $\psi_\Ga=-\infty$ outside $\L$ and Quint \cite{Quint1} showed that $\psi_\Ga$ is a concave upper semi-continuous function satisfying
 $$\cal L=\{\psi_\Ga\ge 0\}\quad \text{ and } \quad \psi_\Ga >0 \quad\text{ on $\inte \L$}.$$

\subsection*{Hausdorff dimension of the directional limit sets}
For a vector $u = (u_1, \cdots, u_k) \in \br_{>0}^k$, a point $(\xi_1, \cdots, \xi_k)\in \cal F$
 is called  a {\it $u$-directional} limit point of $\Ga$ if 
 the geodesic ray 
 $$\{(\xi_1({tu_1}), \cdots, \xi_k({tu_k})):t\ge 0\}$$ accumulates on $\Gamma\ba X$, where
 $\{\xi_i(t):t\ge 0\}$ denotes a unit speed geodesic
 in $X_i = \bH^{n_i}$ toward $\xi_i \in \mathbb S^{n_i-1}$.
 We denote by $$\La_u \subset \La$$ the set of all $u$-directional  limit points of $\Ga$; note that $\La_u$ depends only on the direction of $u$ and it follows easily from the definition of $\La_u$ that $\La_u=\emptyset$ for $u \notin \L$. For $k=1$, the directional limit set is precisely 
 the conical limit set.
 In a higher rank setting,
 the notion of directional limit sets was first considered by Burger \cite{Bu} in the product of two rank one groups and
 then in \cite{BLLO} in general.

We obtain the following estimates on the Hausdorff dimension of directional limit sets in terms of the growth indicator function.
\begin{Thm} \label{main2} Assume that $k \le 3$. 
For any $u = (u_1, \cdots, u_k) \in \inte \L$, we have
 $$ \frac{\psi_\Ga(u)}{\max_i u_i} \le   \dim \La_u \le \frac{\psi_\Ga(u)}{\min_i u_i}. $$
In particular, if $(1, \cdots, 1) \in \inte \L$, then $$\dim \La_{(1, \cdots, 1)} = \psi_{\Ga}(1, \cdots, 1).$$
\end{Thm}

See
Theorem \ref{mainprop} for the upper bound, which is proved for all $k\ge 1$,
and Corollary \ref{upp} for the lower bound.

\subsection*{Symmetric growth indicator functions} 
By the concavity of $\psi_\Ga$ and the strict convexity of the norm ball
$\{\|v\|\le 1\}$, 
there exists a unique unit vector $u_\Ga$, called  the direction of maximal growth, such that
$ \psi_{\Ga}(u_\Ga)=\sup_{\|u\|=1} \psi_\Ga(u).$
By \cite[Coro. III.1.4]{Quint1}, $$\delta = \psi_{\Ga}(u_\Ga).$$

In general, it is hard to determine $u_\Ga$.
However when the growth indicator function $\psi_\Ga$ is symmetric, that is,
it is invariant under all permutations in coordinates,
the concavity of $\psi_\Ga$ implies that $u_\Ga= {1 \over \sqrt{k}}(1, \cdots , 1)$ and
hence by Theorem \ref{main2}, we obtain the following {\it identity}:
\begin{cor}
 If $1\le k\le 3$ and $\psi_\Ga$ is symmetric, then  $u_\Ga= {1 \over \sqrt{k}}(1, \cdots , 1)$ and
\be\label{explicit} \dim \La_{u_\Ga}= {\sqrt k} {\delta}.\ee \end{cor}

If the $\rho_i$ are all conjugate, the growth indicator function $\psi_\Ga$ is symmetric for an obvious reason.  In section \ref{sec.sym}, we construct many geometric examples where
no two of the $\rho_i$ are conjugate to each other and $\psi_\Ga$ is symmetric.

\subsection*{On the proofs}
The class of groups $\Ga$ we consider are precisely those of
 Anosov subgroups of $G$ with respect to a minimal parabolic subgroup in the sense of Guichard and Wienhard \cite{GW}, who generalized Labourie's notion of Anosov representations on Hitchin representations \cite{La}.

One important feature of these Anosov subgroups is that their limit sets consist entirely of conical limit points \eqref{con}. This feature allows us to cover the limit set by shadows using which we can compare the Hausdorff dimension of the limit set and $\delta_{\min}$ and hence prove Theorem \ref{main}. This argument is an easy adaptation of Sullivan's proof on the rank one case.

The upper bound in  Theorem \ref{main2} is obtained by a similar idea and holds for all $k\ge 1$.
The proof of the lower bound in Theorem \ref{main2} is 
based on the reparametrization theorem  for Anosov subgroups (\cite{Samb1}, \cite{BCLS}, \cite{CS}), which provides us with a trivial vector bundle with fiber $\br^{k-1}$ over a compact space $Z$ associated to the dynamics of one-dimensional diagonal flow in the direction of $u$ (see section \ref{fibdyn}). If we denote by $\nu_u$ the Patterson-Sullivan measure for the direction $u$,
then the measure of maximal entropy $\mathsf m_u$ on $Z$ is locally equivalent to
$\nu_u\otimes \nu_u \otimes d \op{Leb}$  and 
 $\mathsf m_u$-almost all points have their $\br^{k-1}$-coordinate map
decaying sublinearly along the flow (Theorem \ref{zo} and Corollary \ref{ae}). Using this,
we get estimates on the local size of $\nu_u$ at almost all points (Theorem \ref{p4}).
For $k\le 3$, we have $\nu_u(\La_u)=1$ by \cite{BLLO}, which enables us
to use the mass distribution principle to prove Theorem \ref{main2}. We remark that this is the exact reason for the hypothesis $k\le 3$ in Theorem \ref{main2}.

Our approach works for any Anosov subgroup of a semisimple real algebraic group of rank at most $3$, provided the Hausdorff dimension of the limit set is computed with respect to a well-chosen metric on the Furstenberg boundary.
The reason we have chosen to write this paper mainly for the product of $\SO^{\circ}(n_i,1)$'s is because $\F$ in this case is simply $\prod_{i=1}^k \mathbb S^{n_i-1}$ and hence is equipped with a {\it natural} metric, that is, the Riemannian product of spherical metrics on $\mathbb S^{n_i-1}$'s. Our theorems are all valid when $\SO^{\circ}(n_i,1)$  is replaced by a rank one simple Lie group
 $G_i=\op{Isom}^\circ X_i$ (here $X_i$ is a rank one Riemannian symmetric space), provided we use a certain sub-Riemannian metric on the Furstenberg boundary $\cal F$ invariant under a maximal compact subgroup of $G$ as described in \cite{Co} (see Remark \ref{rone}).

\begin{Rmk} We mention that a certain upper bound
on the dimension of the limit set for projective Anosov representations 
was obtained in \cite{GMT} 
 and  an equality between the Hausdorff dimension of the limit set and the first simple root critical exponent for certain hyperconvex representations in
$\op{SL}_n(\br)$ was obtained in \cite{PSW2}. 
Both papers neither address the cases of
 products of rank one groups nor yield the identity as in \eqref{in2}, not to mention that the directional limit sets were not considered at all.
 \end{Rmk}

\subsection*{Organization} In section \ref{prelim}, we review basic notions and state known results about Anosov subgroups of $\prod_{i=1}^k \SO^\circ(n_i,1)$. 
In section \ref{LaHdim}, we prove Theorem \ref{main}.  
In section \ref{fibdyn}, we discuss the trivial vector bundle mentioned above, and prove a result that the vector-valued coordinate map associated to $u$ decays with speed $o(t)$ under the time $t$-flow $\exp tu$ (Theorem \ref{zo}).
In section \ref{loc}, we prove Theorem \ref{main2} by studying the local behavior of the measures $\nu_u$ in Theorem \ref{p4}. In the last section \ref{sec.sym}, we discuss some geometric examples with symmetric growth indicator functions.

\subsection*{Question}
As mentioned, our proof for the lower bound
in Theorem \ref{main2} requires the restriction $1\le k\le 3$, whereas the upper bound holds for any $k\ge 1$.
It would be interesting to understand whether lower bound is still valid for a general $k\ge 1$ or not.

\subsection*{Acknowledgements}
Our work has been largely inspired by a pioneering paper of Marc Burger \cite{Bu} on a higher rank Patterson-Sullivan theory. In particular, for $k=2$, the upper bound of Theorem \ref{main2} was already noted in \cite[Thm. 2]{Bu}. We would like to dedicate this paper to him on the occasion of his sixtieth birthday with affection and admiration. We are grateful to Dick Canary for helpful remarks on an earlier version of this paper, and in particular for pointing out how to strengthen our original version of Theorem \ref{main}. We would like to thank Minju Lee for helpful discussions. We are also grateful to anonymous referees for 
many useful comments.

\section{Preliminaries} \label{prelim}
We briefly recall the setup from the introduction.  For each $1 \le i \le k$, $G_i=\op{SO}^\circ(n_i,1)$ and $X_i=(\bH^{n_i},d_i)$ for $n_i \ge 2$.
 Let $G = \prod_{i = 1}^k G_i$, $X=\prod_{i=1}^k X_i$ with $d=\sqrt{\sum d_i^2}$ and
$\cal F=\prod_{i=1}^k \mathbb S^{n_i-1}$.

 Let $\Delta$ be a finitely generated group and $\rho_i : \Delta \to G_i$ a 
 convex cocompact faithful representation with Zariski dense image for each $1\le i\le k$. 
 In the whole paper, let $\G$ be the subgroup of $G$ defined as
$$\G=\{ (\rho_1(\s), \cdots, \rho_k(\s))\in G: \s\in \Delta\}.$$
We will assume that $\Ga$ is Zariski dense, or equivalently, no two $\rho_i$'s are conjugate to each other.

We remark that the class of these groups is precisely
the class of Anosov subgroups of $G$ with respect to a minimal parabolic subgroup  in the sense of Guichard and Wienhard \cite{GW}. This follows from combining \cite[Lem. 3.18, Coro. 4.16 and Thm. 5.15]{GW}: $\Ga$ is Anosov with respect to a minimal parabolic subgroup of $G$ if and only if for each $1\le i \le k$, $\rho_i(\Delta)$ is Anosov with respect to a minimal parabolic subgroup of $G_i$, and Anosov subgroups of $G_i$ are precisely convex cocompact subgroups.

This enables us to use the general theory developed for Zariski dense Anosov subgroups.
Fix a basepoint $o_i\in X_i$ for each $i$, and we write $o=(o_1, \cdots, o_k)\in X$.
Let $$\fa=\br^k\quad\text{ and }\quad \fa^+=\{(u_1, \cdots, u_k)\in \br^k: u_i\ge 0\text{ for all $i$}\}.$$
We denote by $\|\cdot\|$ the standard Euclidean norm on $\fa$.

The {limit set} of $\Gamma$, which we denote by $\Lambda=\Lambda_\Ga$,
is defined as the set of all accumulation points of $\Ga o$ in the Furstenberg boundary $\F $, as in \eqref{dll}. It is the unique $\Ga$-minimal subset of $\F$ (\cite{Ben}, \cite[Lem. 2.13]{LO}).
 
  For each $\xi=(\xi_1, \cdots, \xi_k)\in \F$ and $(t_1, \cdots, t_k)\in \fa$, we write
\be \label{xixi}  \xi(t_1, \cdots, t_k)=(\xi_1(t_1), \cdots, \xi_k(t_k))\ee
where  $\{\xi_i (t):t\ge 0\}$ denotes the unit speed geodesic from $o_i$ to $\xi_i$ in $X_i$. Set
 \be\label{xia} \xi(\fa^+):=\{ \xi(t_1, \cdots, t_k)\in X: t_i\ge 0\text{ for all $i$}\}. \ee   
Recall that $\xi\in \F$ is called a conical limit point if there exists  a sequence $\ga_j\in \G$  such that
$$ \sup_j d(\xi(\fa^+), \ga_j o )<\infty  .$$
If $\La_{\mathsf c}$ denotes the set of all conical limit points, then it is a well-known property of
an Anosov subgroup (cf. \cite[Prop. 7.4]{LO}) that
\be\label{con} \La=\La_{\mathsf c}.\ee
The Cartan projection of $g=(g_i)_{i=1}^k\in G$ is given by
$$\mu(g)=(d_{1}(g_1 o, o), \cdots, d_{k}(g_k o, o) )\in \fa^+.$$
In particular, $d(go, o)=\|\mu(g)\|$.
We denote by $\cal L \subset \fa^+$ the limit cone of $\Ga$, which is the asymptotic cone of $\mu(\Gamma)$. It is a convex cone with non-empty interior \cite{Ben}.
 Let $\delta=\delta_\Ga$ denote the {critical exponent} of $\Ga$,
which is  the abscissa of convergence of the Poincar\'e series $\cal P_\Gamma (s) = \sum_{\ga \in \Ga} e^{-s \|\mu(\ga)\|}$.
It follows from the non-elementary assumption on $\rho_i(\Delta)$ that $\delta>0$.

Let $\psi_\Ga : \fa \to\bb R\cup\{-\infty\}$ denote the growth indicator function of $\Gamma$
defined as in the introduction (see \eqref{grow3}). By the concavity of
$\psi_\Ga$ and the strict convexity of the unit norm ball $\{\|u\|\le 1\}$,  there exists a unique unit vector $ u_\Gamma\in \L $ such that
\be\label{tg} \delta=\sup_{\|u\|=1}\psi_\Ga(u)=\psi_\Ga(u_\Ga) .\ee

As all $\rho_i:\Delta \to G_i$ are faithful convex cocompact, it follows that
 there exist constants $C,C'>0$ such that for all $\sigma\in\Delta$ and $
 1\le i, j\le k$, we have
\begin{equation*}
     \label{eqn.Anosovdef}
d_i (\rho_i(\sigma)o_i, o_i)\ge C d_j (\rho_j(\sigma)o_j, o_j )-C'
\end{equation*}
(cf. \cite[Thm. 5.15]{GW}).
Therefore:
\begin{Thm}\label{p0} We have
 $\L\subset \inte \fa^+ \cup\{0\}$.
\end{Thm}
The following theorem follows from the fact that $\Ga$ is a Zariski dense Anosov subgroup of 
$G = \prod_{i = 1}^k \SO^{\circ}(n_i, 1)$ with respect to a minimal parabolic subgroup  \cite[Lem. 4.8]{Samb3} and \cite[Prop. 4.6 and 4.11]{SP}.  
\begin{Thm} \label{p1} We have  $u_\Ga\in \inte \L$.
\end{Thm}

For $x=(x_1, \cdots, x_k), y=(y_1, \cdots, y_k)\in X$, and $\xi=(\xi_1, \cdots, \xi_k)\in \F$, the $\fa$-valued Busemann function is given as
$$\beta_\xi(x, y)=(\beta_{\xi_1}(x_1, y_1),\cdots, \beta_{\xi_k}(x_k, y_k)) \in \fa$$
where
$\beta_{\xi_i}(x_i, y_i)=\lim_{t\to +\infty}d_{i}(\xi_i(t),
x_i) - d_{i}(\xi_i(t), y_i)$ is the Busemann function on $\mathbb S^{n_i-1}\times X_i \times X_i$.

\begin{Def} For  a linear form $\psi\in \fa^*$, a Borel probability measure $\nu$ on $\La$ is called a $(\Gamma, \psi)$-Patterson-Sullivan measure if the following holds:
for any $\xi\in \La$ and $\ga\in \Ga$,
$$\frac{d\ga_*\nu}{d\nu}(\xi)=e^{-\psi (\beta_\xi(\ga o, o))}$$
where $\ga_*\nu(W)=\nu(\gamma^{-1}W)$ for any Borel subset $W\subset \La$.
\end{Def}

A linear form $\psi\in \fa^*$ is called {\it tangent to $\psi_\Ga$ at $u\in \fa$} if $\psi\ge \psi_\Ga$ and $\psi(u)=\psi_\Ga(u)$.

\begin{Thm}[{\cite[Thm. 7.7 and
Cor. 7.8]{ELO}}, {\cite[Cor. 7.12]{LO}}] \label{p2} Let $u\in \inte \L$.
\begin{enumerate}
    \item  There exists a unique $\psi_u\in \fa^*$ which is tangent to $\psi_\Ga$ at $u$. 
   
    \item There exists a unique $(\Ga,\psi_u)$-Patterson-Sullivan measure, say, $\nu_u$.
\item  The abscissa of convergence of the series $$\cal P_u(s):=\sum_{\ga\in \Ga} e^{-s\psi_u (\mu(\ga))}$$ is equal to $1$ and $\cal P_u(1)=\infty$. 

\end{enumerate}
\end{Thm}

We remark that the existence of $(\Ga, \psi_u)$-Patterson-Sullivan measure was proved by Quint \cite{Quint2}.

\subsection*{Construction of $\nu_u$} Fix $u\in \inte\L$.
By Theorem \ref{p0}, $\L\subset \inte \fa^+ \cup \{0 \}$; this implies that
 all the accumulation points of $\Ga o$ lie in $\F$ and hence in $\La$. Therefore
$\G o\cup \Lambda$ is a compact space.
For $s>1$,  by Theorem \ref{p2}(3), $\cal P_u(s)$ is well-defined
and hence we may consider the probability measure on $\G o \cup  \Lambda$ given by 
\be\label{exp} \nu_{u, s}:= \frac{1}{\cal P_u(s)} \sum_{\ga\in \Ga} e^{-s\psi_u (\mu(\ga))} D_{\ga o} \ee 
where $D_{\ga o }$ denotes the Dirac measure on $\ga o$.
Note that the space of probability measures on $\G o \cup \Lambda$ is
a weak$^*$ compact space.
Therefore, by passing to a subsequence, it weakly converges to a probability measure, say $\tilde \nu_u$, on $\G o\cup \La$. Since $\cal P_u(1)=\infty$ by Theorem \ref{p2}(3), $\tilde \nu_u$ is
supported on $\La$. It is standard to check that $\tilde \nu_u$ is a $(\Ga, \psi_u)$-Patterson-Sullivan measure. Now the uniqueness of $(\Ga, \psi_u)$-Patterson-Sullivan measure (Theorem \ref{p2}(2)) implies that
$\tilde \nu_u=\nu_u$, as given in Theorem \ref{p2}(2). 
Therefore, as $s\to^{+}1$, $\nu_{u, s}$ weakly converges to $\nu_u$.

\subsection*{Hausdorff dimension} 
For $S\subset \F$ and $s>0$, 
the $s$-dimensional Hausdorff measure of $S$ is defined by
$H^s(S) =\lim_{\e\to 0}\inf\{\sum_j r_j^s: S\subset \bigcup_{j\in J} B(x_j, r_j):0<r_j\le \e\}$
where the infimum is taken over all countable covers of $S$ by balls of radius at most $\e$.
The Hausdorff dimension of $S$ is defined as $$\dim S := \inf \{ s \ge 0 : H^s(S) = 0 \},$$
or equivalently the supremum $s$ such that
$H^s(S)=+\infty$.
We refer to \cite{BP} for general facts on Hausdorff dimension.

\section{Hausdorff dimension of $\La$} \label{LaHdim}
In this section, we prove Theorem \ref{m0}, which implies Theorem \ref{main}.

Let  $\delta_{\min}$ denote the abscissa of convergence of the series 
$$ \sum_{\ga
=(\ga_1, \cdots, \ga_k)\in \Ga} e^{-s \min_i d_i (o_i, \ga_i  o_i)}.$$
The notation $\delta_{\rho_i}$ means the critical exponent of $\rho_i(\Delta)$. 

\begin{Thm} \label{m0} For any $k\ge 1$, we have
\be\label{up0} \dim{} \La = \delta_{\min} =\max_{1 \le i \le k} \dim \La_{\rho_i}
=\max_{1 \le i \le k} \delta_{\rho_i}
 .\ee \end{Thm}

We need to introduce some notations for the proof of this theorem. These notations will also be used in the next sections as well.
Let $K_i\simeq \SO(n_i) $ be the maximal compact subgroup of $G_i$ given as the stabilizer of $ o_i \in X_i$.  Fixing a unit tangent vector $\mathsf{v}_i$ at $o_i$, let $ M_i:=\op{Stab}(\mathsf{v}_i) $. We then have the following identification:
 $G_i/K_i=X_i$ and $G_i/M_i=\T^1 X_i$.
 Let $A_i=\{a^{(i)}_{t}:t\in \br\} <G_i$ denote the one-parameter subgroup of semisimple elements whose right translation action on $G_i/M_i$ corresponds to
the geodesic flow on $\T^1 X_i$. 
Set 
$K=\prod_{i = 1}^k  K_i<G$,
$M=\prod_{i = 1}^k  M_i<G$, and $A=\prod_{i = 1}^k A_i$.
We also set $A^+=\prod_{i = 1}^k A_i^+$ 
where $A_i^+=\{a^{(i)}_{t}:t\ge 0\}$.
Then $X=G/K$.

 For each $i$ and $g_i\in G_i$, we denote by
 $g_i^+\in \mathbb S^{n_i-1}$ and $g_i^-\in \mathbb S^{n_i-1}$ respectively 
 the forward and backward endpoints of the geodesic determined by the tangent vector $[g_i]\in \T^1 \bH^{n_i}$. For $g=(g_1, \cdots, g_k)\in G$,
 we set
 \be\label{gpm} g^{\pm}=(g_1^{\pm}, \cdots, g_k^{\pm})\in \F .\ee 

\medskip 

The shadows are important tools in our proof:
\begin{Def}[Shadows] For $R>0$ and $x\in X$,
the shadow $O_R(o, x)$ is defined as
\be \label{shadow} O_R(o, x)=\{\eta\in \cal F: \exists g\in K, a \in A^+ \text{ s.t } g^+=\eta \text{ and }
d(g a o, x)\le R \}. \ee
\end{Def}

\subsection*{Proof of Theorem \ref{m0}} For $\xi\in \F$ and
$r>0$, let $B(\xi, r)$ denote the ball in $\F=\prod_{i = 1}^k \S^{n_i-1}$ centered at $\xi$ of radius $r$.
For $g=(g_i)$, we write 
$$\min \mu(g) = \min d_i(o_i, g_i o_i).$$

For each $N \in \mathbb N$, let $\La_N:=\La\cap \limsup_{\ga\in \Ga} O_N(o, \ga o)$, that is, $$\La_N = \{ \xi \in \La : \exists \ga_{\ell} \to \infty \mbox{ in } \Ga \mbox{ such that } \xi \in O_N(o, \ga_\ell o) \mbox{ for all } \ell \ge 1 \}.$$ 
There exists a constant $c_N > 0$ such that
for any $\ga \in \Ga$, 
the shadow $O_N(o, \ga o)$ is contained in a ball
$B(\xi_{\ga}, c_N e^{-\min \mu(\ga)})$ for some
$\xi_\ga\in \F$; in particular, the diameter of $O_N(o, \ga o)$ is at most
$2c_N e^{-\min \mu(\ga)}$.

It follows from Theorem \ref{p0} that if a sequence $\ga_i\in \Ga$ tends to infinity,
then $\min \mu(\ga_i) \to \infty$.
Hence, for any fixed $t > 0$, we have \be\label{lambda}
\La_N \subset \bigcup_{\ga \in \Ga, \min \mu(\ga) > t} O_N(o, \ga o)
.\ee

Let $s>\delta_{\min}$ be any number. Since $\sum_{\ga \in \Ga} e^{- s \min \mu(\ga)} < \infty$, 
$$\lim_{t\to \infty} \sum_{\ga \in \Ga, \min \mu(\ga) > t} e^{- s \min \mu(\ga)}=0.$$ 

Since $\{ O_N(o, \ga o) : \ga \in \Ga, \min \mu(\ga) > t\}$
is a cover of $\La_N$ with subsets whose diameter is at most
$2c_Ne^{-\min \mu(\ga)}$, this implies that the $s$-dimensional Hausdorff measure of $\La_N$
is equal to zero. Since $s>\delta_{\min}$ is arbitrary, we get
 $$\dim \La_N \le \delta_{\min}.$$  Since $\La $ is equal to the conical limit set $\La_{\mathsf c}$ by \eqref{con}, we have $\La = \cup_{N \in \N} \La_N$. Consequently, 
\be\label{cons} \dim \La  \le \sup_{N\in \mathbb N} \dim \La_N \le \delta_{\min}.\ee 

Since, for any $s>0$, 
$$\sum_{\ga \in \Ga} e^{-s \min \mu(\ga)} \le \sum_{i = 1}^k \sum_{\sigma \in \Delta} e^{-s d_i(\rho_i(\sigma)o, o)},$$
the series $\sum_{\ga \in \Ga} e^{-s \min \mu(\ga)}$ converges when $s > \max_i \delta_{\rho_i}.$ It implies that 
\be \label{dmin}
\delta_{\min} \le \max_i \delta_{\rho_i}.
\ee 
Since the projection map $\La \to \La_{\rho_i}$ is Lipschitz, we have
\be \label{lip}
\max_i \dim \La_{\rho_i} \le \dim \La.
\ee
By combining \eqref{cons}, \eqref{dmin}, \eqref{lip},
and the fact that $\dim \La_{\rho_i} = \delta_{\rho_i}$ by Sullivan \cite{Su}, 
we have $$\max_i \dim \La_{\rho_i} \le \dim \La \le \delta_{\min}\le  \max_i  \delta_{\rho_i}
=\max_i \dim \La_{\rho_i} .$$   This finishes the proof.

\section{Fibered dynamical systems and $\ker\psi_u$-coordinate map} \label{fibdyn}
Recall that for $u \in \inte \L$, $\psi_u$ denotes the unique linear form tangent to $\psi_{\Ga}$ at $u$, as was given in Theorem \ref{p2}).
In this section, for each direction $u$ in the interior of the limit cone,
we discuss  the $\ker \psi_u$-trivial bundle over a compact space $Z$, associated to the dynamics of one-dimensional diagonal flow in the direction $u$, and show that
the $\ker\psi_u$-coordinate map $\hat K_u(z,t)$ of this bundle decays sublinearly as time
$t\to \infty$ for almost all $z$ with respect to the measure $\mathsf{m}_u$ of maximal entropy (see Theorem \ref{zo} and Corollary \ref{ae}). The results in this section will be used as main ingredients of our estimates on the Hausdorff dimension of directional limit sets in section \ref{loc}.

We continue to use notation 
$M_i, K_i, A_i, M, K, A, g^{\pm}$ from the last section.
 Now for each $1\le i\le k$, the map $$[g_i]\to (g_i^+, g_i^-, \beta_{g_i^+}(o_i, g_i o_i))$$ gives an $\SO^{\circ}(n_i, 1)$-equivariant homeomorphism between 
$\SO^{\circ}(n_i, 1) /M_i$ and $\{(\xi_i, \eta_i)\in \S^{n_i-1}\times \S^{n_i - 1}: \xi_i\ne \eta_i\} \times \br$,
where the left $\SO^{\circ}(n_i, 1)$-action on the latter space is given by
$$g_i \cdot (\xi_i,\eta_i,s)=(g_i\xi_i,g_i\eta_i,s+\beta_{\xi_i}(g_i^{-1}o_i,o_i)) .$$

This homeomorphism is called the Hopf parametrization of 
$\SO^{\circ}(n_i, 1)/M_i$ under which the right $A_i$-action on $\SO^{\circ}(n_i, 1)/M_i$
corresponds to the translation flow on $\br$.

For $\xi\in \F=\prod_{i=1}^k \mathbb S^{n_i-1}$,
we write $\xi_i$ for its $i$-th component.
We set $\F^{(2)}=\{(\xi, \eta)\in \cal F\times \cal F: \xi_i\ne\eta_i\text{ for all $i$}\}$.
Then the Hopf parametrization of $\SO^{\circ}(n_i, 1)/M_i$
extends to the Hopf parametrization of $G/M$ componentwise, and gives the $G$-equivariant homeomorphism 
 $G/M \simeq \cal F^{(2)}\times \fa$ given by
$$[g]\to ( g^+, g^-, \beta_{g^+}(o, g o))\quad\text{where} \quad  g^{\pm}=(g_i^{\pm}).$$
 
 Set  $\La^{(2)}=\F^{(2)}\cap (\La\times \La)$. Then $\Omega:=\Ga\ba (\La^{(2)}\times \fa)$ is identified with the closed subspace  $\{[g]\in \Ga\ba G/M: g^{\pm}\in \La\} $ of $\Ga\ba G/M$ via the Hopf parameterization.
 
\subsection*{Trivial $\ker \psi_u$-vector bundle}

 We fix a unit vector $u\in \inte\L$ in the rest of this section.
Consider the $\Ga$-action on the space
$\La^{(2)}\times\bb R$ by
$$\ga \cdot (\xi,\eta,s)=(\ga\xi,\ga\eta,s+\psi_u(\beta_{\xi}(\ga^{-1}o,o))).
$$ 

The reparametrization theorems for Anosov groups (\cite[Prop. 4.1]{BCLS}, \cite[Thm. 4.15]{CS})
imply that $\Ga$ acts properly discontinuously and cocompactly on $\La^{(2)}\times\bb R$. Hence $Z:=\Ga\ba (\La^{(2)}\times\bb R)$ is a compact space.
Now the $\G$-equivariant projection 
$\La^{(2)}\times \fa \to\La^{(2)}\times \br$ given  by 
$(\xi, \eta, v)\mapsto (\xi, \eta, \psi_u(v))$
induces an affine bundle with fiber $\ker \psi_u$: 
$$\pi: \Omega=\Ga\ba (\La^{(2)}\times \fa) \to Z=\Ga\ba (\La^{(2)}\times \br).$$

It is well-known that such a bundle is indeed a trivial vector bundle, and hence we can choose 
a continuous global section $$\mathsf s:Z\to \Omega$$ so that $\pi \circ \mathsf s =\op{id}_Z$. Denote by $\{\tau_t: t\in \br\}$ the flow on $Z$ given by 
translations by $t$ on $\br$.  For $v=(v_1, \cdots, v_k)\in \fa$, 
we write 
$$a_{v}=(a_{v_1}^{(1)}, \cdots, a_{v_k}^{(k)})\in A.$$ 
 \begin{Def}[$\ker\psi_u$-coordinate map] We define a continuous $\ker\psi_u$-valued map
 $$\hat K_u :Z\times  \bb R \to \ker \psi_u$$
as follows:
for $z\in Z$ and $t\in \br$,
\begin{equation}\label{eq.K} 
    \mathsf s(z)a_{tu} =\mathsf s(z\tau_t)a_{\hat K_u(z,t)}.
\end{equation}
\end{Def} 

Let $\mathsf{m}_u$ denote the $\psi_u$-Bowen-Margulis-Sullivan measure on $Z$; that is,
$\mathsf{m}_u$ is the unique $\tau_t$-invariant probability measure on $Z$  which is locally equivalent to $\nu_u\otimes \nu_u \otimes ds$. It follows from \cite{BCLS} that $\mathsf{m}_u$ is the measure of maximal entropy and in particular ergodic for the $\tau_t$-flow.
\begin{Thm}\label{zo}
For $\mathsf m_u$-a.e. $z\in Z$, we have
\be\label{lim}
\lim_{t\to\infty}\frac{1}{t}\hat K_u(z,t)=0.
\ee 
\end{Thm}
\begin{proof}

Combining the reparametrization theorem \cite[Prop. 4.1]{BCLS} and \cite[Prop. 3.5]{Samb},  we deduce that
there exists a H\"older continuous function $F : Z\to\ker\psi_u$ with $\int_{Z} F\,d\mathsf{m}_u=0$ such that
for all $z\in Z$ and $t\in \br$,
\be\label{hatk}  \hat K_u(z,t)=\int_0^t F(z\tau_s)\,ds + E(z)- E(z\tau_t)\ee  for some bounded measurable function $E:Z\to \ker \psi_u$. 
The Birkhoff ergodic theorem for the $\tau_s$ flow on $(Z, \mathsf{m}_u)$ implies that for $\mathsf{m}_u$-almost all $z\in Z$, we have
$$\lim_{t\to \infty} \frac{1}{t} \int_0^t F(z\tau_s)\,ds= \int_Z F d\mathsf{m}_u =0;
$$
 hence 
$$
\lim_{t\to\infty}\frac{1}{t}\hat K_u(z,t)=0
$$
since $E$ is bounded. 
\end{proof}

Fix a compact subset $D\subset G/M$ such that $\mathsf s(Z)=\Ga\ba\Ga D$, and for each $z\in Z$, write $\mathsf s(z)=\Ga \tilde{\mathsf s}(z)$ for some $\tilde{\mathsf s}(z)\in D$. Hence 
\be\label{every} \La^{(2)}\times \fa =\Gamma D a_{ \ker \psi_u}.\ee
We will sometimes consider $D$ as a right $M$-invariant subset of $G$  by abuse of notation.

\begin{lem} \label{lem.jan19}
 For any $g\in G$
with $g^{\pm}\in \La$, there exist $z_g\in Z$ and $w_g\in \ker \psi_u$ such that for all $t\in \br$, there exists $\ga_{g,t}\in \Ga$ satisfying
\be\label{gt} \ga_{g, t} g a_{tu}=  \tilde{\mathsf s}(z_g\tau_t) a_{\hat K_u(z_g,t)+w_g}.\ee 
\end{lem}
\begin{proof} By \eqref{every},
there exist $\ga\in\Ga$, $z\in Z$ and $w\in\ker\psi_u$ such that
$\ga g=\tilde{\mathsf s}(z)a_{w}$, and hence
$$
\ga ga_{tu}=\tilde{\mathsf s}(z)a_{tu+w}.
$$
On the other hand, by \eqref{eq.K}, there exists $\ga_{z, t}\in\Ga$ such that
$$
\ga_{z,t}\tilde{\mathsf s}(z)a_{tu}=\tilde{\mathsf s}(z\tau_t) a_{\hat K_u(z,t)}.
$$
Therefore,
$$
\ga_{z,t}\ga ga_{tu}=\tilde{\mathsf s}(z\tau_t) a_{\hat K_u (z,t) +w}.
$$
It remains to set $\ga_{g,t}=\ga_{z,t}\ga$.\end{proof}

 For each $g \in G$ with $g^{\pm} \in \La$, we choose $z_g \in Z$ and $w_g \in \ker \psi_u$ as given by Lemma \ref{lem.jan19}.
We also set  $K^\dagger_u (g,t):= \hat K_u (z_g,t) +w_g\in \ker \psi_u$, so that for all $t\in \br$,
\be\label{gaa} \ga_{g,t} g a_{tu}\in D  a_{K^\dagger_u (g,t)}.\ee

\begin{cor}\label{ae}
For $\nu_u$-a.e. $\xi \in \La$, there exists $\La(\xi) \subset \La$ with $\nu_u(\La(\xi)) = 1$ such that for any $g \in G$ with $g^+ = \xi$ and $g^- \in\La(\xi)$, we have 
\be\label{kug}
\lim_{t\to\infty}\frac{1}{t}  K^\dagger_u (g,t)=0.
\ee 
\end{cor}
\begin{proof}
Since $\mathsf m_u$ is equivalent to $\nu_u \otimes \nu_u \otimes ds$, Theorem \ref{zo} implies that for $\nu_u$-a.e. $\xi \in \La$, there exists a $\Ga$-invariant measurable subset $\La(\xi) \subset \La$ with $\nu_u(\La(\xi)) = 1$ such that $\lim_{t \to \infty} \frac{1}{t} \hat K_u([(\xi, \eta, s)], t) = 0$ for any $\eta \in \La(\xi)$ and $s \in \R$. For each $\xi$ satisfying this, let $g \in G$ with $g^+ = \xi$ and $g^- \in \La(\xi)$. It suffices to show that $g$ satisfies \eqref{kug}.
Let $z=[(\xi_1, \eta_1, s)]\in Z$ be such that
$\gamma_{g,0} g  =\tilde s(z) a_{K^\dagger(g, 0)}$ as given by Lemma \ref{lem.jan19} and \eqref{gaa}.
 It follows that $g^+ = \xi \in \Ga \xi_1$ and $g^- \in \Ga \eta_1$, and hence $z_g = [(\xi, \eta_1, s)]$ for $\eta_1 \in \La(\xi)$. Therefore, $\hat K_u(z_g, t)/t \to 0$ as $t \to \infty$. Since $K^{\dagger}_u(g, t) - \hat K_u(z_g, t) = w_g$ which is independent of $t$, we get
 $ K^\dagger_u (g,t)/t \to 0$ as $t\to \infty$.

\end{proof}

Let $\m^{\BMS}_u$ denote the Bowen-Margulis-Sullivan measure
on $\Omega\subset \Ga\ba G$ given by $\m^{\BMS}_u=\m_u\otimes \op{Leb}|_{\ker \psi_u}\otimes dm$ where $dm$ denotes the Haar measure on $M$; this is an $A$-invariant ergodic (infinite) Radon measure, as shown in \cite{LO2}. We also remark that by \cite{BLLO},
$\m^{\BMS}_u$  is $\{a_{tu}:t\in \br\}$-ergodic if and only if $k\le 3$.
In terms of this measure, Corollary \ref{ae} can be formulated as the following which may be regarded as an analogue of Sullivan's result \cite[Coro. 19]{Su}, which predates his logarithm law.

\begin{Thm} For $\m^{\BMS}_u$-a.e. $x\in \Gamma\ba G$,
we have $$\lim_{t\to \infty }\frac{ d(x a_{tu} o, o)}{t}=0 .$$
\end{Thm}

Since $D$ is compact, this theorem follows from Corollary \ref{ae} in view of \eqref{gaa}.

\section{Hausdorff dimension of $\La_u$ and local behavior of $\nu_u$}\label{loc}

For each $u=(u_1, \cdots, u_k)\in \fa^+$, 
 the $u$-directional  limit set $\La_u\subset \La$ is defined
 as $$\La_u := \{ \xi \in \F : \liminf_{t \to +\infty} d(\xi(tu_1,\cdots, tu_k), \Ga o) < \infty \}$$ 
 where $\xi(\cdot)$ is defined as in \eqref{xixi}.

In this section, we 
 obtain estimates on $\dim \La_u$ for $u\in \inte \L$. 
 We will obtain an upper bound for $\dim \La_u$ for any $k\ge 1$ but our
 lower bound is obtained only when $k\le 3$; the main reasons are
 that \begin{enumerate}
     \item the lower bound
 is deduced from local estimates on $\nu_u $ (Theorem \ref{p4}) using the mass distribution principle and
    \item the directional limit set $\La_u$ has positive
 $\nu_u$-measure if and only if  $k\le 3$ (Theorem \ref{dic}).
 \end{enumerate}

In the whole section, we fix a unit vector  
$$u = (u_1, \cdots, u_k)\in \inte \L.$$ 
However, note that the statements below still hold for an arbitrary vector in $\inte\L$
since all quantities are homogeneous. We also set
\be\label{MU} M_u=\max_{1\le i\le k}  u_i,\, m_u=\min_{1\le i\le k} u_i,\text{ and }
\delta_u:=\psi_u(u) = \psi_{\Ga}(u)>0.\ee 

\subsection*{Upper bound for dimension}
For any $N\in \mathbb N$,
set $$\Ga_{N} (u):=\{\ga\in \Ga:
\|\mu(\ga) -t_\ga u\|\le N\text{ for some $t_\ga>0$}\} $$
and
 $$\La_N^*(u) :=\limsup_{t \to \infty} \bigcup_{\ga \in \Ga_N(u), \| \mu(\ga) \| \ge t} O_N(o, \ga o) ,$$ 
 where $O_N(o,\ga o)$ is a shadow defined as in \eqref{shadow}.
 
 We will use the following simple observation:
 \begin{lem}\label{LUC} We have
 $$\La_u\subset \bigcup_{N\in \mathbb N} \La_N^*(u).$$
 \end{lem}
 
 \begin{proof}
 Let $\xi \in \La_u$. Choose any $g\in K$  such that
 $g^+=\xi$. 
 Then $\xi(tu)=g a_{tu} o$, $t\ge 0$, is a geodesic ray toward $\xi$.
 By the definition of $\La_u$, there exist $N> 0$ and sequences
 $t_\ell \to \infty$, $\ga_\ell \in \Ga$ such that $d(\xi(t_\ell u), \ga_\ell o ) \le  N$ for all $\ell\ge 1$. 
 Note that the absolute value of each component of $\mu(\ga_{\ell}) - t_{\ell}u$ is bounded by $N $, and hence $\|\mu(\ga_{\ell}) - t_{\ell}u\| \le kN $. Replacing $N$ with $kN $, if necessary, 
 we may assume
 that $\ga_\ell \in \Ga_N(u)$ for all $\ell\ge 1$. By the definition of shadows
 in \eqref{shadow}, it follows
 that $\xi \in O_N(o, \ga_\ell o)$ for all $\ell\ge 1$. 
 As $\ga_\ell \in \Ga_N(u)$, we have
 $\lVert \mu(\ga_\ell) \rVert \ge \lVert u \rVert t_\ell - N  $ and hence $\mu(\ga_\ell)\to \infty$ as $\ell \to \infty$.  Therefore
 $\xi\in \La_N^*(u)$; this completes the proof.
 \end{proof}

 \begin{Thm} \label{mainprop} For any $k\ge 1$, we have
\be\label{up}  \quad \dim{} \La_u \le \frac{\delta_u}{m_u} 
 .\ee 
\end{Thm}
\begin{proof} Fix $N\in \mathbb N$.
For each $\ga\in \Ga_N(u)$, we fix $t_\ga>0$ such that
$\|\mu(\ga) -t_\ga u\|\le N$, which exists by the definition of $\Ga_N(u)$.
 Then there exists $d_N>0$ such that for any $\ga\in  \Ga_{N}(u)$,
the shadow $O_N(o, \ga o)$ is contained in a ball
$B(\xi_{\ga}, d_N e^{-t_\ga m_u})$ for some
$\xi_\ga\in \F$.
Since $\|\mu (\ga) -t_\ga u\|\le N$, by applying $\psi_u$,
we get $$|\psi_u(\mu(\ga))\delta_u^{-1}-t_\ga | \le N\delta_u^{-1} \|\psi_u\|_{\op{op}}$$ where $\|\psi_u\|_{\op{op}}$ denotes the operator norm of $\psi_u$. 

Therefore, for some constant $d_N'\ge 1$, we have that for all
$\ga\in \Ga_N(u)$,
we have
$$O_N(o, \ga o) \subset B(\xi_{\ga}, d_Ne^{-t_\ga  m_u}) \subset  B\left(\xi_{\ga}, d_N'e^{-m_u\delta_u^{-1} \psi_u(\mu(\ga))}\right);$$
in particular, the diameter of $O_N(o, \ga o)$ is at most
$2d_N'e^{-m_u\delta_u^{-1} \psi_u(\mu(\ga))}$.
Moreover, for any $t>1$,
$\{ O_N(o, \ga o) : \ga \in \Ga_{N}(u), \|\mu(\ga)\|\ge t\}
$ is a cover of $\La_N^*(u)$.

Let $s > \delta_u/m_u$ be any number.
By Theorem \ref{p2}(3), 
we have $$\lim_{t\to \infty} \sum_{\ga\in \Ga, \|\mu(\ga)\|\ge t} e^{-sm_u\delta_u^{-1} \psi_u(\mu(\ga))}=0 .$$
It implies that
 the $s$-dimensional Hausdorff measure of $\La_N^*(u)$ is zero. 
 Since $s > \delta_u/m_u$ is arbitrary, it follows that
 $$\dim{} \La_N^*(u) \le \frac{\delta_u}{m_u}.$$
Since $$\La_u\subset \bigcup_{N\in \mathbb N} \La_N^*(u)$$ by Lemma \ref{LUC}, this implies the desired bound:
$\dim{} \La_u \le \frac{\delta_u}{m_u}.$ 
\end{proof}
\begin{rmk}
We can replace $B(\xi_{\ga}, d_N e^{-t_{\ga}m_u})$ with $e^{t_{\ga} (M_u - m_u) \sum_{i = 1}^k (n_i - 1)}$ balls of radius $d_Ne^{-t_{\ga}M_u}$. We then have the upper bound $$\dim \La_u \le \frac{\delta_u + (M_u - m_u) \sum_{i = 1}^k (n_i-1)}{M_u}$$ which is smaller than the upper bound in \eqref{up} when $m_u \sum_{i = 1}^k (n_i - 1) < \delta_u$.
\end{rmk}

\subsection*{The local size of $\nu_u$}

We define the following subset of $\La$:
\be \label{newdef}
\La^*_u = \left\{ \xi \in \La : \begin{matrix}
\exists \La(\xi) \subset \La \mbox{ with } \nu_u(\La(\xi)) = 1 \mbox{ such that } \\ \lim_{t \to \infty} \frac{1}{t}K^{\dagger}_u(g, t) = 0 \\ \mbox{ for any } g \in G \mbox{ with } g^+ = \xi \mbox{ and } g^- \in \La(\xi)
\end{matrix} \right \} .\ee
Note that $\La_u^*$ is not necessarily a subset of $\La_u$.

 By Corollary \ref{ae}, we have
 $$\nu_u(\La_u^*)=1.$$

We will be using the following two lemmas.

\begin{lem}  \label{com}  There exists a compact subset $\cal S\subset G$ such that for any $\xi \in \La$ and for any measurable subset $\La'\subset \La$  with $\nu_u(\La') = 1$,  there exists $g\in \cal S$ such that $g^+=\xi$ and $g^-\in \La'$.
\end{lem} 

\begin{proof} This lemma is proved in
 \cite[Lem. 10.6]{LO} for $\La'=\La$.
 It suffices to replace $\cal S$ by the one-neighborhood of $\cal S$, say, $\cal S_0$. Let $\xi\in \La$, and $g\in \cal S$ be
 such that $g^+=\xi$ and $g^-\in \La$. Then we can find a neighborhood
 $\cal O$ of $g^-$ such that for any $\eta\in \cal O$, there exists $h\in \cal S_0$ such that $ h^+=\xi$ and $ h^-=\eta .$
 Since $\nu_u(\La')=1$, we have $\La'$ is dense in $\La$, and hence
 $\La'\cap O\ne \emptyset$. This implies the claim.
\end{proof}

Note that the proof of this lemma can be extended to general Anosov subgroups as it only uses  \cite[Lem. 10.6]{LO}.

The following shadow lemma is obtained for any $\Ga$-conformal measure of
 any discrete Zariski dense subgroup $\Ga<G$:
\begin{lemma}[Shadow lemma]\cite[Lem. 7.8]{LO} \label{sss}
There exists $R_0>0$ such that for all $R>R_0$, there exists $c=c(\psi_u, R) \ge 1$ such that for any  $\ga\in \G$,
\be\label{sha} c^{-1} \cdot  e^{-\psi_u(\mu(\ga))} \le \nu_u (O_R(o, \ga o))\le c \cdot  e^{-\psi_u(\mu(\ga))}.\ee 
\end{lemma}

For $r>0$, let $B(\xi_i, r)$ denote the ball in 
$\S^{n_i - 1}$
centered at $\xi_i$ of radius $r$.
The following theorem is one of two key ingredients of our proof for the  lower bound of $\dim \La_u$ (Corollary \ref{upp}):
\begin{Thm}\label{p4} Let $k\ge 1$. There exists $C_1, C_2>0$ such that for any $\xi=(\xi_1, \cdots, \xi_k)\in \La_u^*$,
 and for any sufficiently small $\e>0$, there exists $t_0=t_{\e,\xi}>0$ such that for all $t\ge t_0$,
\be\label{conetwo}  C_1 \cdot e^{-\delta_u (1+\e)t  } \le  \nu_u\left( \prod_{i=1}^k B(\xi_i, e^{-u_i t}) \right)  \le C_2 \cdot e^{-\delta_u (1-\e)t  } .\ee 
\end{Thm}
\begin{proof}
Choose $g\in \cal S$ such that $g^+=\xi$ and $g^-\in \La(\xi)$ where $\La(\xi)$ is given in \eqref{newdef} and $\cal S$ is a compact subset of $G$ given in Lemma \ref{com}.
Let $\e>0$.
By the definition \eqref{newdef} of $\La_u^*$,
there exists $t_0=t_{\e,g} >0$ such that for each $1\le i\le k$,
the absolute value of the $i$-th component of $K^{\dagger}(g, t)\in \fa=\br^k$ is
$$\mbox{at most} \quad \frac{\e u_i t}{4}\quad\text{for all } t > t_0.$$

Recall the definition of $\ga_{g,t}$ from \eqref{gaa}:
$\gamma_{g,t}g a_{tu}= d_t a_{K^\dagger (g, t)} $ where $d_t\in D$.
Therefore $\ga_{g,t}^{-1}=g a_{tu - K^{\dagger}(g,t)} d_t^{-1}$. Let $q$ be the diameter of $D^{-1}o$.

Note that there exists $c_0>0$ such that  for all $t>t_0$, 
$$  O_{1}(o, \ga_{g,t}^{-1}o )\subset
O_{q+1}(o, ga_{tu - K^{\dagger}(g,t)}o)\subset  \prod_{i=1}^k B(\xi_i, c_0 e^{-u_i(1-  \e/4) t}); $$
the first inclusion is immediate from the definition of the shadows.
Hence we deduce from Lemma \ref{sss} that for all $t>\max(t_0, 2\log c_0/(u_i\e))$,
\begin{multline*}
\beta  \cdot e^{-\delta_u  t} \le  \nu_u\left( \prod_{i=1}^k B(\xi_i, 
c_0 e^{-u_i(1-\e/4) t}) \right) \le  \nu_u\left( \prod_{i=1}^k B(\xi_i, 
e^{-u_i(1-\e/2) t}) \right)
\end{multline*} 
for some constant $\beta=\beta(\psi_u, D)>0$. By reparametrizing $(1-\e/2)t=s$,
this implies the lower bound in \eqref{conetwo}.

On the other hand, for all $t \ge t_0$,
$$\begin{aligned}
\prod_{i = 1}^k B(\xi_i, e^{-u_i(1+\varepsilon/4)t}) & \subset  
 O_{p}(o, ga_{ (1+\e/4) tu }o) \\
 & \subset O_p(o, ga_{tu-K^{\dagger}(g, t)}o) \subset O_{p+q}(o, \ga_{g, t}^{-1}o)
 \end{aligned}$$ where $p$ depends only on $\cal S$. Hence, by \eqref{sha},
 $$\nu_u \left( \prod_{i = 1}^{k} B(\xi_i, e^{-u_i (1 + \varepsilon/4)t}) \right) \le c e^{-\psi_u(\mu(\ga_{g, t}^{-1}))}$$ where $c=c(\psi_u, p+q)$. Recalling that $\ga_{g, t}^{-1} = ga_{tu - K^{\dagger}(g, t)}d_t^{-1}$, 
 we have $$\|\mu(\ga_{g,t}^{-1}) -(tu-K^{\dagger}(g,t))\| \le \|\mu(g) \| +\|\mu(d_t)\| \le \beta'$$
 where $\beta'=2 \max\{ \|\mu(h)\|: h\in \cal S\cup  D\}$.
 
Since $K^{\dagger}(g,t)\in \ker \psi_u$, we have for all $t>t_0,$
  $|\psi_u(\mu(\ga_{g,t}^{-1})) -t\delta_ u |\le \|\psi_u\|_{\op{op}} \beta'$.
 Therefore we have $$\nu_u \left( \prod_{i = 1}^{k} B(\xi_i, e^{-u_i (1 + \varepsilon/4)t}) \right) \le C_2 e^{-\delta_u t}$$ where $C_2>0$ depends only on $\cal S$, $D$ and $\psi_u$. 
 In other words, for all $t>2t_0$,
  $$\nu_u \left( \prod_{i = 1}^{k} B(\xi_i, e^{-u_i t}) \right) \le C_2 e^{-(1-\e) \delta_u t}.$$ This proves the upper bound in \eqref{conetwo}.
\end{proof}

\subsection*{Lower bound for dimension}
The second key ingredient of the proof of Corollary \ref{upp} is
the following recent result:
\begin{Thm} \cite[Thm. 1.6]{BLLO} \label{dic} 
We have
\begin{equation*}
\nu_u(\La_u)=\begin{cases} 1&\text{ if $k\le 3$}\\ 0 &\text{otherwise.}
 \end{cases} \end{equation*} \end{Thm} 

This together with Corollary \ref{ae} implies:
\begin{cor}\label{ust}
 If $k\le 3$, then $\nu_u(\La_u^*\cap \La_u)=1$.
\end{cor}
We are now ready to prove the following lower bound on $\dim \La_u$:
\begin{cor}\label{upp} For $k\le 3$,  we have
 $$ \dim \La_u\ge  \dim{} (\La^*_u\cap \La_u) \ge \frac{\delta_u}{M_u}$$
 where $\delta_u$ and $M_u$ are given in \eqref{MU}.
\end{cor}
\begin{proof}
Recall that $B(\xi, r)$ denotes the ball of radius $r > 0$ centered at $\xi = (\xi_1, \cdots, \xi_k)$ in $\F = \prod_{i = 1}^k \S^{n_i - 1}$ with respect to the Riemannian metric. Since $u_i/M_u \le 1$,  we have that for all $t>0$
$$B(\xi, e^{-t}) \subset \prod_{i = 1}^k B(\xi_i, e^{-u_i t / M_u}).$$

Fix $\e>0$.
Therefore Theorem \ref{p4} implies that
there exists $C>0$, independent of $\e>0$,  such that for any
$\xi\in \La_u^*$ and  for all sufficiently small $r=r_{\e, \xi}>0$,
 $$\nu_u (B (\xi, r)) \le C \cdot r^{(1-\e) \delta_u/M_u } .$$ 
Since $\nu_u (\La_u^* \cap \La_u) = 1$ by Corollary \ref{ust},
the Mass distribution property (more precisely,
 Rogers-Taylor theorem  \cite[Theorem 4.3.3]{BP}) now implies that 
 $$\dim (\La_u^*\cap \La_u)  \ge (1-\e) \delta_u/M_u  .$$
  Since $\e>0$ is arbitrary, this proves the claim. 
 \end{proof}

\begin{Rmk} \label{rone} Our proofs of Theorems \ref{main} and \ref{main2} work in the same way for the product 
 $G=\prod_{i=1}^k G_i$ where $G_i=\op{Isom}^\circ (X_i)$ is a simple Lie group for a 
 Riemannian symmetric space $X_i$ of rank one. The Furstenberg
 boundary $\cal F$ of $G$ is the product $\prod_{i=1}^k
 \partial X_i$ of geometric boundaries of $X_i$, and the Hausdorff dimension of the limit set
 of $\Ga<G$ is to be computed with respect to a certain sub-Riemannian metric on $\F$ which is invariant under a maximal compact subgroup of $G$, as described in \cite{Co}. In these situations, shadows are comparable to metric balls by \cite[Thm. 2.2]{Co} and $\dim \La_{\rho_i}=\delta_{\rho_i}$ by
 \cite[Thm 6.1]{CI}. Given these, the discussions in sections \ref{prelim}--\ref{loc} remain valid.
 \end{Rmk}
 
\section{Examples of symmetric growth indicator functions} \label{sec.sym}
 Given a self-joining subgroup
 $\G<G$, there doesn't seem to be any general method
 to compute the maximal growth
 direction.  In this  section,
 we provide a class of geometric examples of $\Ga$ whose
 growth indicator functions are symmetric, and hence whose
  maximal growth direction $u_\Ga$ is parallel to $(1, \cdots, 1)$.

Let $\Delta$ be a finitely generated group, and
$\Out \Delta$ denote its outer automorphism group,
i.e., the group of automorphisms of $\Delta$ modulo the inner automorphisms.
Note that, for a representation $\rho:\Delta\to \so$ and $\iota\in\Out\Delta$, $\rho\circ \iota $ is well-defined up to conjugation in $\so$.
\begin{lem}\label{extra}  Let $k\ge 2$.
Let $\rho_1:\Delta\to \so$ be a non-elementary convex cocompact faithful representation and $\iota \in \Out \Delta$ be of order $k$. Let $\rho_i = \rho_1 \circ \iota^{i-1}$ for $2 \le i \le k$ and let $\Ga_{\iota} := (\prod_{i=1}^k
 \rho_i)(\Delta)$.  Then 
 $$\psi_{\Ga_\iota}=\psi_{\Ga_\iota}\circ \theta
\;\;  \text{ and } \;\; u_{\Ga_\iota}=\tfrac{1}{\sqrt k}(1, \cdots, 1)$$
 where $\theta$ denotes the cyclic permutation $(x_1, \cdots, x_k) \mapsto (x_2, \cdots, x_k, x_1)$.
 \end{lem}
 
\begin{proof}
For each $1 \le n \le k$, let $\Ga^{(n)} = (\prod_{i=1}^k \rho_i \circ \iota^n)(\Delta)$. Since $\iota^k = 1$ in $\Out \Delta$, $\Ga^{(n)}$ can be regarded as a group obtained by permuting coordinates in a cyclic way. Hence,
\be \label{eq.symmetry}
\L_{\Ga^{(n)}}  = \theta(\L_{\Ga^{(n-1)}})\;\; \text{ and }\;\; 
\psi_{\Ga^{(n)}}  = \psi_{\Ga^{(n-1)}} \circ \theta^{-1}.
\ee 

However, $\Ga^{(n)} = \Ga_{\iota}$ for all $n$; since applying an automorphism to all coordinates does not change the group. Hence, \eqref{eq.symmetry} implies that $\L_{\Ga_{\iota}}$ and $\psi_{\Ga_{\iota}}$ are invariant under the cyclic permutation $\theta$ of coordinates.
\end{proof}

\subsection*{Examples in $\bH^2\times \bH^2$}
Let us describe some examples to which Lemma \ref{extra} can be applied. We begin in dimension 2.
For a closed orientable surface $S$ of genus $g\ge 2$, one can obtain homeomorphisms $\iota:S\to S$ of order 2 in a number of ways. Figure \ref{fig.iota} indicates how this can be done: Arrange the surface in $\R^3$ so that it is symmetric by a $180^\circ$ rotation. There are several possibilities distinguished by the number of intersection points of the surface with the rotation axis, which yield fixed points of $\iota$.

\begin{figure}[ht]
		\centering
		\begin{tikzpicture}[scale=1.0, every node/.style={scale=0.7}]
		
		\begin{scope}[shift={(-3.75, 0)}, yscale=0.75]
		    \draw (-2, 0) .. controls (-2, 1) and (-1, 0.5) .. (-0.75, 0.5) .. controls (-0.5, 0.5) and (-0.3, 0.75) .. (0, 0.75) .. controls (0.3, 0.75) and (0.5, 0.5) .. (0.75, 0.5) .. controls (1, 0.5) and (2, 1) .. (2, 0);
		    
		    \draw[thick, violet] (-2, 0) .. controls (-2, 0.2) and (-1.4, 0.2) .. (-1.4, 0);
		    \draw[thick, violet, dashed] (-2, 0) .. controls (-2, -0.2) and (-1.4, -0.2) .. (-1.4, 0);

			\draw (-1.5, 0.1) .. controls (-1.4, -0.1) and (-1, -0.1) .. (-0.9, 0.1);
			\draw (-1.4, 0) .. controls (-1.3, 0.1) and (-1.1, 0.1) .. (-1, 0);

			\begin{scope}[rotate=180]
    			\draw (-2, 0) .. controls (-2, 1) and (-1, 0.5) .. (-0.75, 0.5) .. controls (-0.5, 0.5) and (-0.3, 0.75) .. (0, 0.75) .. controls (0.3, 0.75) and (0.5, 0.5) .. (0.75, 0.5) .. controls (1, 0.5) and (2, 1) .. (2, 0);

			    \draw (-1.5, -0.1) .. controls (-1.4, 0.1) and (-1, 0.1) .. (-0.9, -0.1);
				\draw (-1.4, 0) .. controls (-1.3, -0.1) and (-1.1, -0.1) .. (-1, 0);
				
		        \draw[thick, violet, dashed] (-2, 0) .. controls (-2, 0.2) and (-1.4, 0.2) .. (-1.4, 0);
		        \draw[thick, violet] (-2, 0) .. controls (-2, -0.2) and (-1.4, -0.2) .. (-1.4, 0);

			\end{scope}
			
		    \begin{scope}
			    \begin{scope}[shift={(1.2, 0)}]
				    \draw (-1.5, 0.1) .. controls (-1.4, -0.1) and (-1, -0.1) .. (-0.9, 0.1);
				    \draw (-1.4, 0) .. controls (-1.3, 0.1) and (-1.1, 0.1) .. (-1, 0);
			    \end{scope}
			\end{scope}
			
			\begin{scope}[rotate=45]
			\filldraw[white] (0.56, 0.56) circle(2pt);
			\filldraw[white] (0.05, 0.05) circle(2pt);
			\draw[->] (1.5, 1.5) arc(45:380:0.2);
			\filldraw[white] (1.45, 1.5) circle(2pt);
			\draw[thick] (1.75, 1.75) -- (0, 0);
		    \draw[thick, dotted] (-0.1, -0.1) -- (-0.5, -0.5);
		    \draw[thick] (-0.6, -0.6) -- (-1.75, -1.75);
		    
		    \draw (1.6, 1.5) node[right] {$180^{\circ}$};
		    \end{scope}

		\end{scope}
		
		\begin{scope}[yscale=0.75, xscale=1.2, rotate=90]
		    \draw (-2, 0) .. controls (-2, 1) and (-1, 0.5) .. (-0.75, 0.5) .. controls (-0.5, 0.5) and (-0.3, 0.65) .. (0, 0.65) .. controls (0.3, 0.65) and (0.5, 0.5) .. (0.75, 0.5) .. controls (1, 0.5) and (2, 1) .. (2, 0);

		    \begin{scope}[shift={(0.3, 0)}, scale=1.2]
			    \draw (-1.5, 0.05) .. controls (-1.4, -0.1) and (-1, -0.1) .. (-0.9, 0.05);
			    \draw (-1.4, 0) .. controls (-1.3, 0.1) and (-1.1, 0.1) .. (-1, 0);
			\end{scope}
			
			\begin{scope}[shift={(0, 0)}]
			    \draw[violet, thick] (0, 0.65) .. controls (0.2, 0.65) and (0.2, 0.1) .. (0, 0.1);
			\draw[violet, thick, dashed] (0, 0.65) .. controls (-0.2, 0.65) and (-0.2, 0.1) .. (0, 0.1);
			\end{scope}
			
			\begin{scope}[shift={(1.2, 0)}]
			\end{scope}

			\begin{scope}[rotate=180]
    			\draw (-2, 0) .. controls (-2, 1) and (-1, 0.5) .. (-0.75, 0.5) .. controls (-0.5, 0.5) and (-0.3, 0.65) .. (0, 0.65) .. controls (0.3, 0.65) and (0.5, 0.5) .. (0.75, 0.5) .. controls (1, 0.5) and (2, 1) .. (2, 0);
    			
    			\begin{scope}[shift={(0.3, 0)}, scale=1.2]
			        \draw (-1.5, -0.05) .. controls (-1.4, 0.1) and (-1, 0.1) .. (-0.9, -0.05);
				    \draw (-1.4, 0) .. controls (-1.3, -0.1) and (-1.1, -0.1) .. (-1, 0);
				\end{scope}
				
				\draw[violet, thick, dashed] (0, 0.65) .. controls (0.2, 0.65) and (0.2, 0.05) .. (0, 0.05);
		    	\draw[violet, thick] (0, 0.65) .. controls (-0.2, 0.65) and (-0.2, 0.08) .. (0, 0.08);
				
			\end{scope}
			
		    \begin{scope}[scale=1.2]
			    \begin{scope}[shift={(1.2, 0)}]
			        \draw (-1.5, 0.05) .. controls (-1.4, -0.1) and (-1, -0.1) .. (-0.9, 0.05);
				    \draw (-1.4, 0) .. controls (-1.3, 0.1) and (-1.1, 0.1) .. (-1, 0);
				    
			    \end{scope}
			\end{scope}
			
			\begin{scope}[xscale=0.75]
		    \draw[->] (3.3, 0) arc(0:335:0.2);
			\filldraw[white] (3.3, 0) circle(2pt);
		    \draw (3.3, -0.15) node[right] {$180^{\circ}$};
		    \end{scope}
		    
			\begin{scope}
			\draw[thick] (1.75*1.414, 0) -- (2, 0);
			\draw[thick, dotted] (2, 0) -- (1.4, 0);
			\draw[thick] (1.35, 0) -- (0.9, 0);
			\draw[thick, dotted] (0.9, 0) -- (0.2, 0);
			\draw[thick] (0.2, 0) -- (-0.2, 0);
			\draw[thick, dotted] (-0.2, 0) -- (-0.9, 0);
			\draw[thick] (-0.9, 0) -- (-1.35, 0);
			\draw[thick, dotted] (-1.4, 0) -- (-2, 0);
			\draw[thick] (-2, 0) -- (-1.75*1.414, 0);
		    \end{scope}

		\end{scope}
		
		
		\begin{scope}[shift={(3.75, 0)}]
		    
		    \draw (-2, 0) .. controls (-2, 1) and (-1, 0.5) .. (-0.75, 0.5) .. controls (-0.5, 0.5) and (-0.3, 0.65) .. (0, 0.65) .. controls (0.3, 0.65) and (0.5, 0.5) .. (0.75, 0.5) .. controls (1, 0.5) and (2, 1) .. (2, 0);
		   
			\draw[thick, violet] (-2, 0) .. controls (-2, 0.15) and (-1.5, 0.15) .. (-1.5, 0);
		    \draw[thick, violet, dashed] (-2, 0) .. controls (-2, -0.15) and (-1.5, -0.15) .. (-1.5, 0);
		    
		    

			\begin{scope}[rotate=180]
			    \draw (-2, 0) .. controls (-2, 1) and (-1, 0.5) .. (-0.75, 0.5) .. controls (-0.5, 0.5) and (-0.3, 0.65) .. (0, 0.65) .. controls (0.3, 0.65) and (0.5, 0.5) .. (0.75, 0.5) .. controls (1, 0.5) and (2, 1) .. (2, 0);
			    
				\draw[thick, violet, dashed] (-2, 0) .. controls (-2, 0.15) and (-1.5, 0.15) .. (-1.5, 0);
		    \draw[thick, violet] (-2, 0) .. controls (-2, -0.15) and (-1.5, -0.15) .. (-1.5, 0);
			\end{scope}
			
		    \begin{scope}[rotate=0]
			    \begin{scope}[shift={(1.2, 0)}]
				    \draw (-1.5, 0) .. controls (-1.5, 0.2) and (-0.9, 0.2) .. (-0.9, 0) .. controls (-0.9, -0.2) and (-1.5, -0.2) .. (-1.5, 0);
			    \end{scope}
			\end{scope}
			
			\begin{scope}[shift={(-1.2, 0)}]
			    \begin{scope}[rotate=0]
			        \begin{scope}[shift={(1.2, 0)}]
				        \draw (-1.5, 0) .. controls (-1.5, 0.2) and (-0.9, 0.2) .. (-0.9, 0) .. controls (-0.9, -0.2) and (-1.5, -0.2) .. (-1.5, 0);
			        \end{scope}
			    \end{scope}
			\end{scope}
			
			\begin{scope}[shift={(1.2, 0)}]
			    \begin{scope}[rotate=0]
			        \begin{scope}[shift={(1.2, 0)}]
				        \draw (-1.5, 0) .. controls (-1.5, 0.2) and (-0.9, 0.2) .. (-0.9, 0) .. controls (-0.9, -0.2) and (-1.5, -0.2) .. (-1.5, 0);
			        \end{scope}
			    \end{scope}
			\end{scope}

		\end{scope}

		\begin{scope}[shift={(3.75, 0)}, yscale=0.75, rotate=45]
		
			\draw[->] (1.5, 1.5) arc(45:380:0.2);
			\filldraw[white] (1.45, 1.5) circle(2pt);
			\draw[thick] (1.75, 1.75) -- (0.65, 0.65);
			\draw[thick, dotted] (0.6, 0.6) -- (0.15, 0.15);
			\draw[thick] (0.14, 0.14) -- (-0.115, -0.115);
		    \draw[thick, dotted] (-0.15, -0.15) -- (-0.6, -0.6);
		    \draw[thick] (-0.6, -0.6) -- (-1.75, -1.75);
		    
		    \draw (1.6, 1.5) node[right] {$180^{\circ}$};
		    
		\end{scope}

		\end{tikzpicture}
		
	\caption{\small{Examples of involutions $\iota \in \Out \pi_1(S)$ where $S$ is of genus $3$. Indicated curves are mapped to each other by $\iota$.}} \label{fig.iota}
		\end{figure}
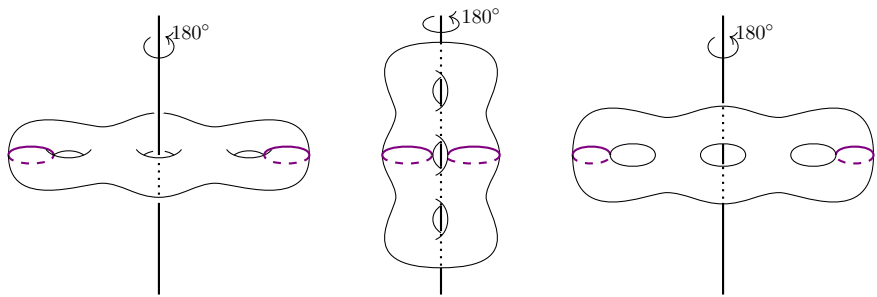

In order for the example $(\rho,\rho\circ \iota)$ not to be trivial, we need the representations not to be conjugate in $\SO^{\circ}(2,1).$ That is, $\rho$ should not represent a point of Teichm\"uller space $\cal{T}(S)$ which is fixed by $\iota$.
This is always possible when $g\ge 3$; to see this, note that there are disjoint, non-homotopic simple closed curves exchanged by $\iota$ in each case. They can be assigned different lengths by a hyperbolic structure, which would then not be fixed by $\iota$.  In genus 2, one just needs to avoid the hyperelliptic involution -- the one with 6 fixed points -- which fixes every point in $\cal{T}(S)$. All other rotations will do. 

\subsection*{Examples in $\bH^3\times \bH^3$}
Examples involving 3-manifolds are also plentiful. Consider for example a ``book of
$I$-bundles'' constructed as follows (see Anderson-Canary \cite{AC}). Let $S_1,\cdots,S_{\ell}$ be ${\ell}$ copies of a surface of
genus $g\ge 1$  with one boundary component and let $Y$ be the 2-complex obtained by
identifying all the boundary circles to one. A choice of cyclic order $c$ on the ${\ell}$ surfaces
determines a thickening of $Y$ to a 3-manifold $N_c$: form $S_i\times [-1,1]$ for each
$i$, and identify the annulus $\partial S_i\times[0,1]$ with $\partial S_j \times [-1,0]$
whenever $j$ follows $i$ in the order $c$ (the identification should take $[0,1]\to
[-1,0]$ by an orientation-reversing homeomorphism, and should respect the original
identification of the boundary circles). See Figure \ref{book}.

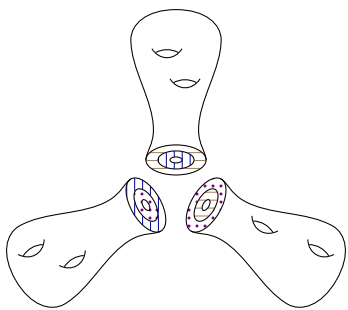
\begin{figure}[ht]
    \centering
	
	\begin{tikzpicture}[scale=0.4, every node/.style={scale=0.4}]
		
		\begin{scope}[shift={(-1, 0)}]
		\begin{scope}[rotate=30]
		\begin{scope}[xscale=0.5]
		
	    	\draw[pattern = vertical lines, pattern color = blue] (0, 0) circle(1);
    		\draw[fill=white, draw=white] (0, 0) circle(0.6);
    		\draw[fill=white, pattern = dots, pattern color = violet] (0, 0) circle(0.6);
    		\draw[fill=white] (0, 0) circle(0.2);
		
    		\draw (0, 1) .. controls (-0.5, 1) and (-1, 0.75) .. (-2, 0.75) .. controls (-4, 0.75) and (-6, 1.5) .. (-8, 1.5) .. controls (-9, 1.5) and (-10, 1) .. (-10, 0);
    		\begin{scope}[yscale=-1]
        		\draw (0, 1) .. controls (-0.5, 1) and (-1, 0.75) .. (-2, 0.75) .. controls (-4, 0.75) and (-6, 1.5) .. (-8, 1.5) .. controls (-9, 1.5) and (-10, 1) .. (-10, 0);
    		\end{scope}
		
    		\begin{scope}[shift={(-1, 0.5)}]
    		    \draw (-8, 0.2) .. controls (-7.5, -0.2) and (-6.5, -0.2) .. (-6, 0.2);
    		    \draw (-7.75, 0.075) .. controls (-7.25, 0.2) and (-6.75, 0.2) .. (-6.25, 0.075);
		    \end{scope}
		    \begin{scope}[shift={(1, -0.5)}]
    		    \draw (-8, 0.2) .. controls (-7.5, -0.2) and (-6.5, -0.2) .. (-6, 0.2);
    		    \draw (-7.75, 0.075) .. controls (-7.25, 0.2) and (-6.75, 0.2) .. (-6.25, 0.075);
    		\end{scope}
		
		\end{scope}
		\end{scope}
		\end{scope}
		
		\begin{scope}[shift={(0, 1.5)}]
		\begin{scope}[rotate=-90]
		\begin{scope}[xscale=0.5]
		
	    	\draw[pattern = horizontal lines, pattern color = orange!40!gray] (0, 0) circle(1);
    		\draw[fill=white, draw=white] (0, 0) circle(0.6);
    		\draw[fill=white, pattern = vertical lines, pattern color = blue] (0, 0) circle(0.6);
    		\draw[fill=white] (0, 0) circle(0.2);
		
    		\draw (0, 1) .. controls (-0.5, 1) and (-1, 0.75) .. (-2, 0.75) .. controls (-4, 0.75) and (-6, 1.5) .. (-8, 1.5) .. controls (-9, 1.5) and (-10, 1) .. (-10, 0);
    		\begin{scope}[yscale=-1]
        		\draw (0, 1) .. controls (-0.5, 1) and (-1, 0.75) .. (-2, 0.75) .. controls (-4, 0.75) and (-6, 1.5) .. (-8, 1.5) .. controls (-9, 1.5) and (-10, 1) .. (-10, 0);
    		\end{scope}
		
    		\begin{scope}[shift={(-1, 0.5)}]
    		\begin{scope}[shift={(-7, -0.3)}]
    		\begin{scope}[xscale=2, yscale=0.5]
       		\begin{scope}[rotate=90]
            \begin{scope}[shift={(6, -0.5)}]
    		    \draw (-8, 0.2) .. controls (-7.5, -0.2) and (-6.5, -0.2) .. (-6, 0.2);
    		    \draw (-7.75, 0.075) .. controls (-7.25, 0.2) and (-6.75, 0.2) .. (-6.25, 0.075);
    		\end{scope}
    		\end{scope}
    		\end{scope}
    		\end{scope}
		    \end{scope}
		    \begin{scope}[shift={(-1, 0.5)}]
    		\begin{scope}[shift={(-5, 0.3)}]
    		\begin{scope}[xscale=2, yscale=0.5]
       		\begin{scope}[rotate=90]
            \begin{scope}[shift={(6, -0.5)}]
    		    \draw (-8, 0.2) .. controls (-7.5, -0.2) and (-6.5, -0.2) .. (-6, 0.2);
    		    \draw (-7.75, 0.075) .. controls (-7.25, 0.2) and (-6.75, 0.2) .. (-6.25, 0.075);
    		\end{scope}
    		\end{scope}
    		\end{scope}
    		\end{scope}
		    \end{scope}

		\end{scope}
		\end{scope}
		\end{scope}
		
		\begin{scope}[shift={(1, 0)}]
		\begin{scope}[rotate=-210]
		\begin{scope}[xscale=0.5]
		
	    	\draw[pattern = dots, pattern color = violet] (0, 0) circle(1);
    		\draw[fill=white, draw=white] (0, 0) circle(0.6);
    		\draw[fill=white, pattern = horizontal lines, pattern color = orange!40!gray] (0, 0) circle(0.6);
    		\draw[fill=white] (0, 0) circle(0.2);
		
    		\draw (0, 1) .. controls (-0.5, 1) and (-1, 0.75) .. (-2, 0.75) .. controls (-4, 0.75) and (-6, 1.5) .. (-8, 1.5) .. controls (-9, 1.5) and (-10, 1) .. (-10, 0);
    		\begin{scope}[yscale=-1]
        		\draw (0, 1) .. controls (-0.5, 1) and (-1, 0.75) .. (-2, 0.75) .. controls (-4, 0.75) and (-6, 1.5) .. (-8, 1.5) .. controls (-9, 1.5) and (-10, 1) .. (-10, 0);
    		\end{scope}
		
    		\begin{scope}[yscale=-1]
    		\begin{scope}[shift={(-1, 0.5)}]
    		    \draw (-8, 0.2) .. controls (-7.5, -0.2) and (-6.5, -0.2) .. (-6, 0.2);
    		    \draw (-7.75, 0.075) .. controls (-7.25, 0.2) and (-6.75, 0.2) .. (-6.25, 0.075);
		    \end{scope}
    		\begin{scope}[shift={(3, 0.25)}]
    		    \draw (-8, 0.2) .. controls (-7.5, -0.2) and (-6.5, -0.2) .. (-6, 0.2);
    		    \draw (-7.75, 0.075) .. controls (-7.25, 0.2) and (-6.75, 0.2) .. (-6.25, 0.075);
    		\end{scope}
    		\end{scope}
		
		\end{scope}
		\end{scope}
		\end{scope}

	\end{tikzpicture}

    \caption{\small{Book of $I$-bundles with three surfaces and patterns indicating the identification.}} \label{book}

\end{figure}

The result $N_c$ is homotopy equivalent to $Y$, and has ${\ell}$ boundary components of genus
$2g$. It admits many convex cocompact hyperbolic structures: it is easy to construct one
``by hand'' by attaching Fuchsian structures along the common boundary using the
Klein-Maskit combination theorem \cite{Mas}.  The Ahlfors-Bers theory parametrizes all
convex cocompact representations as the Teichm\"uller space of $\partial N_c$ (cf. \cite{Mar}). A permutation of $(1,\cdots,{\ell})$ induces a homeomorphism of $Y$ which extends to a
homotopy equivalence of $N_c$ which, if the permutation does not preserve or reverse the cyclic order,
will not correspond to a homeomorphism. Selecting such a permutation of order 2, we have
an automorphism that cannot be an isometry for any hyperbolic structure on $N_c$. (Even if it does correspond to a homeomorphism
one can choose the hyperbolic structure on $N_c$ using a point in ${\mathcal T}(\partial
N_c)$ that is not symmetric with respect to the involution).

\end{document}